\documentclass[twoside]{amsart}
\usepackage{amsaddr}
\usepackage{graphicx}
\usepackage[normalem]{ulem}

\usepackage{graphicx}
\usepackage{graphics}
\usepackage{xifthen}

\usepackage[inline]{enumitem}
\setenumerate{leftmargin=*}

\usepackage{amsmath,amssymb}
\usepackage{amsthm}
\usepackage{tikz-cd}
\usetikzlibrary{decorations.pathreplacing}
\usepackage{braket}
\usepackage{mathtools}
\usepackage{mathrsfs}

\theoremstyle{plain}
\newtheorem{theorem}{Theorem}[section]

\newtheorem{lemma}[theorem]{Lemma}
\newtheorem*{lemma*}{Lemma}
\newtheorem{proposition}[theorem]{Proposition}
\newtheorem{corollary}[theorem]{Corollary}
\newtheorem*{corollary*}{Corollary}

\newtheorem*{claim*}{Claim}
\newtheorem*{theorem*}{Theorem}

\newtheoremstyle{break}%
{}{}%
{\itshape}{}%
{\bfseries}{}
{\newline}{}

\theoremstyle{break}

\theoremstyle{definition}
\newtheorem{definition}[theorem]{Definition}
\newtheorem*{definition*}{Definition}
\newtheorem{example}[theorem]{Example}

\theoremstyle{remark}

\newtheorem{obs*}[theorem]{Observation}
\newtheorem*{remark}{Remark}
\newtheorem{remark*}[theorem]{Remark}

\newcommand{\N}{\mathbb{N}}
\newcommand{\Z}{\mathbb{Z}}
\newcommand{\R}{\mathbb{R}}

\newcommand{\itembf}[2]{%
\ifthenelse{\isempty{#1}}{\item {\bfseries #2}}{\item[\textbf{#1}]{\bfseries #2}}%
}

\newcommand{\dfn}{\mathrel{\mathop:}=}

\newcommand{\sqed}{\hfill $/\!\!/$}

\newcommand{\scratch}[1]{{\color{red}\textbf{#1}}}

\DeclareMathAlphabet{\mathbbold}{U}{bbold}{m}{n}
\DeclareMathAlphabet{\cmb}{OT1}{cmbr}{m}{n}

\DeclareMathOperator{\img}{im}

\DeclareMathOperator{\Map}{Map}
\DeclareMathOperator{\Homeo}{Homeo}

\DeclareMathOperator{\inter}{int}

\DeclareMathOperator{\diam}{diam}

\DeclareMathOperator{\stab}{stab}

\DeclareMathOperator{\Ends}{Ends}

\DeclareMathOperator{\asdim}{asdim}

\newlist{clist}{enumerate}{1}
\setlist[clist]{label=(\roman*),itemsep=0ex,
  topsep=.5ex,parsep=0ex,leftmargin=3.5em}

\newlist{citem}{itemize}{1}
\setlist[citem,1]{label=\textbullet,topsep=1ex,itemsep=1ex,parsep=0ex}

\ifdefined\clean
\renewcommand{\scratch}[1]{{}}
\else
\fi

\usepackage{url}
\usepackage[colorlinks,citecolor=blue,linkcolor=red!80!black]{hyperref}
\usepackage[nameinlink, capitalize]{cleveref}

\AddToHook{env/lemma/begin}{\crefalias{theorem}{lemma}}
\AddToHook{env/proposition/begin}{\crefalias{theorem}{proposition}}
\AddToHook{env/remark*/begin}{\crefalias{theorem}{remark*}}
\AddToHook{env/example/begin}{\crefalias{theorem}{example}}
\AddToHook{env/definition/begin}{\crefalias{theorem}{definition}}
\AddToHook{env/corollary/begin}{\crefalias{theorem}{corollary}}

\crefname{theorem}{Theorem}{Theorems}
\crefname{lemma}{Lemma}{Lemmas}
\crefname{proposition}{Proposition}{Propositions}
\crefname{remark*}{Remark}{Remarks}
\crefname{example}{Example}{Examples}
\crefname{definition}{Definition}{Definitions}
\crefname{corollary}{Corollary}{Corollaries}

\newcommand{\A}{\mathcal A}
\newcommand{\G}{\mathcal G}
\newcommand{\cc}{\mathcal C}

\newcommand{\M}{\mathcal M}
\newcommand{\X}{\mathscr X}
\newcommand{\xc}{{\vphantom{\mathscr X}\smash{\hat{\mathscr X}}}}

\DeclareMathOperator{\PMap}{PMap}

\DeclareMathOperator{\Id}{Id}

\newcounter{COMMENTS}
\newcommand{\newComment}[3]{
    \expandafter\newcommand\csname #1\endcsname[1]{ 
        \textbf{%
		\color{#3}(\uppercase{#2}\theCOMMENTS)%
        }
        \marginpar{\scriptsize\raggedright\textbf{%
			{\color{#3}(\uppercase{#2}\theCOMMENTS)#1: 
	}} ##1}
	    \stepcounter{COMMENTS}
    }
    \expandafter\newcommand\csname #2Annot\endcsname[1]{{\color{#3} ##1}} 
}

\newComment{George}{gs}{blue}
\newComment{Michael}{mk}{teal}

\begin{document}

\title[Geometric models for surface mapping class groups]
{Geometric models and asymptotic dimension for 
  infinite-type surface mapping class
groups}
\author{Michael C.\ Kopreski, George Shaji}

\begin{abstract}
  Let $S$ be an infinite-type surface and let $G \leq \Map(S)$ be a
  locally bounded Polish subgroup.
  We construct a metric graph $\M$ of simple arcs and
  curves on $S$ preserved by the action of 
  $G$ and for which the vertex orbit map $G \to V(\M)$ is a coarse
  equivalence; if $G$ is boundedly generated, then $\M$ is a
Cayley--Abels--Rosendal graph for $G$ and the orbit map is a
quasi-isometry. In particular, if $S$ contains a non-displaceable
subsurface and $G \geq \PMap_c(S)$ is boundedly generated or 
$G \in \{\overline{\PMap_c(S)},
  \PMap(S), \Map(S) \}$ and is locally bounded,
  then $\asdim \M = \asdim G = \infty$.
  This result completes
  the classification of the asymptotic dimension of stable
  boundedly generated infinite-type surface mapping class groups begun by
  Grant--Rafi--Verberne.
\end{abstract}

\maketitle

\section{Introduction and main results}\label{sec:intro}

Let $S$ be a surface of infinte topological type.  A
\textit{(metric) arc and curve model} for 
$G \leq \Map(S)$ is a connected (metric) 
graph whose vertices are collections of (possibly intersecting)
simple arcs and curves on
$S$, with an isometric action of $G$ induced by the
permutation of its vertices.  

\begin{theorem}\label{thm:model} 
  Let $S$ be an infinite-type surface and let $G \leq \Map(S)$ be a locally bounded
  Polish subgroup.
  \begin{enumerate}
    \item There exists a metric arc and curve model $\M$ for
  $G$ for which the orbit map restricted to $V(\mathcal{M})$ is a continuous 
  coarse equivalence.
    \item  
  If additionally $G$ is boundedly generated, then $\M$
  is a Cayley--Abels--Rosendal graph for $G$ and the
  orbit map 
  is a continuous quasi-isometry.
  \end{enumerate}
\end{theorem}

In particular, the coarse equivalence and quasi-isometry types of
$G \leq \Map(S)$ are described by a (metric) arc and curve model, whenever
they are
well-defined. A compact subsurface $\Delta \subset S$ 
is \textit{non-displaceable by $G$} if
there exists no $f \in G \leq \Map(S)$ such that
$\Delta \cap f\Delta = \varnothing$.  
From \cref{thm:model} we obtain: 

\begin{theorem}\label{thm:asdim}
 Let $S$ be an infinite-type surface and $G\leq \Map(S)$ a Polish
 subgroup  with a
 non-displaceable subsurface and containing $\PMap_c(S)$. If
 $G$ is boundedly generated or $G \in \{\overline{\PMap_c(S)},
 \PMap(S),\Map(S) \}$ and locally bounded, then $\asdim G = \infty$. 
\end{theorem}

\cref{thm:asdim} answers 
\cite[Qn.~1.8]{shift} of Grant--Rafi--Verberne 
 and completes their characterization of 
the asymptotic dimension of stable boundedly generated
infinite-type surface mapping class groups.

\begin{corollary}\label{cor:asdim}
  For stable $S$ with boundedly generated $\Map(S)$, 
  $\asdim \Map(S) = \infty$ if and
  only if $S$ has a non-displaceable subsurface or an essential
  shift; otherwise, $\Map(S)$ is coarsely bounded and $\asdim \Map(S)
  = 0$.
\end{corollary}

To our knowledge, our construction obtains the first examples of an
arc and curve model
admitting a geometric (\v Svarc--Milnor-type) action of the
mapping class group of an arbitrary infinite-type surface; see
\cite{anschel} for a construction of curve graphs for translateable
surfaces. 

\subsection{Outline}\label{sec:outline}

In \cref{sec:prelim}, we state some known results about the
asymptotic dimension of mapping class groups of infinite-type
surfaces, specifically from \cite{shift}, following which we
introduce some background and relevant tools for the coarse
geometry of Polish groups from \cite{rosendal} and \cite{bdhl}. 
In \cref{sec:ccpt}, we describe witness-cocompactness, a key tool for
computing asymptotic dimension, and sketch the main theorem in
\cite{ga_asdim}: 

\begin{theorem}\label{thm:ccpt_asdim} Let $S$ be an
  infinite-type surface and let 
  $\mathcal{M}$ be a witness-cocompact arc and curve model for
  $\PMap_c(S)$.
  Then $\asdim \mathcal{M} = \infty$. 
\end{theorem}

\noindent
In \cref{sec:ccar}, we introduce \textit{coarse
Cayley--Abels--Rosendal graphs}, which extend the
Cayley--Abels--Rosendal graphs of \cite{bdhl} for 
locally bounded Polish groups; in particular, we prove a \v
Svarc--Milnor-type result (\cref{prop:ccar}).
\cref{sec:model_general} constructs the model $\M$ satisfying \cref{thm:model}.
Finally, \cref{sec:nondisp} proves \cref{thm:asdim}.

\subsection{Acknowledgements}\label{sec:acknowledge}

The authors would like to thank Thomas Hill for acquainting them with
Cayley--Abels--Rosendal graphs and Robbie Lyman, George Domat and Brian Udall for
helpful discussions. Additionally, the authors thank Mladen Bestvina and Priyam Patel for their suggestion to  generalize the work to the
locally bounded case.  This work was supported by NSF grant no.\
2304774 and no.\ 1840190: \textit{RTG: Algebra, Geometry, and Topology at
the University of Utah}.

\section{Preliminaries}\label{sec:prelim}

We first review Rosendal's work on Polish topological groups and
introduce the necessary background on coarse structures from
\cite{rosendal}. We then recall Cayley--Abels--Rosendal graphs for
topological groups {\cite{bdhl}} and several facts on the topology
of boundedly generated mapping class groups. Lastly, we summarize
the relevant results from \cite{shift}.

\subsection{Coarse structure}\label{sec:coarse}

The Polish groups considered herein are typically not finitely or
compactly generated.  Nonetheless, following Rosendal
 we may associate to every topological group $G$ 
a canonical left-invariant \textit{coarse structure},
which generalizes the (quasi)geometric structure classically
associated to a group.  This coarse structure will permit a
well-defined
 coarse equivalence and quasi-isometry
type for \textit{locally bounded} and \textit{boundedly generated}
Polish groups, respectively (\cref{sec:cbgen}).  


\begin{definition}[{\cite[Defn.~2.2]{rosendal}}]\label{def:coarse_str}
    A \textit{coarse structure} on a set $X$ is a collection $\mathcal{E}$ of
    subsets $E\subseteq X\times X$ satisfying the following:

    \begin{itemize}
        \item The diagonal $\{(x,x)\mid x\in X\}$ is in $\mathcal{E}$.

        \item If $F\in\mathcal{E}$ and $E\subseteq F$, then $E\in\mathcal{E}$

        \item if $E,F\in\mathcal{E}$, then $E\cup F,E^{-1},E\circ
      F\in\mathcal{E}$, where $E^{-1}=\{(y,x)\mid(x,y)\in E\}$ and
      $E\circ F:=\{(x,z)\mid\exists y\in X, (x,y)
      \in F,(y,z)\in E\}$.
  \end{itemize}
\end{definition}

\begin{example}[{\cite[Example ~2.3]{rosendal}}]\label{ex:cs_pseudometric}
    The simplest examples of coarse structures arise from
    pseudometrics on a set $X$: Given a pseudometric $d$ on $X$, we
    may define the coarse structure induced by $d$ as follows:
   \[ 
    \mathcal{E}_d:=\{E\subset X\times X\mid E\subseteq E_\alpha\text{ for some }\alpha<\infty\}
  \] 
    where $E_\alpha:=\{(x,y)\mid d(x,y)<\alpha\}$.
\end{example}

\begin{definition}
[{\cite[Defn.~2.12]{rosendal}}]\label{def:cb_set} A subset
$A\subseteq X$ of a coarse space $(X,\mathcal{E})$ is said to be
\textit{coarsely bounded} if $A\times A\in\mathcal{E}$. 
\end{definition}

A topological group has a canonical left-invariant coarse structure:

\begin{definition}[{\cite[Defn.~2.10]{rosendal}}]\label{def:left_coarse_str}
    For a topological group $G$, the \textit{left-coarse structure}
    $\mathcal{E}_L$ is defined by 
    \[ 
    \mathcal{E}_L:=\bigcap\{\mathcal{E}_d\mid d\text{ is a
    continuous, left-invariant pseudometric on G}\}\:.
  \] 
\end{definition}
\noindent
For a topological group $G$ with its left-coarse structure
$\mathcal{E}_L$, we denote by $\mathcal{CB}$ the collection of all
coarsely bounded subsets of $G$.
This collection
is actually an ideal of sets additionally 
closed under the operations of topological closure, inversion and products.

\begin{definition}[{\cite[Defn.~2.10]{rosendal}}]\label{def:cs_ideal}
    Given an ideal $\A$, we can define the coarse structure
    $\mathcal{E_A}$ on the group $G$ as follows:
    \[
    \mathcal{E_A}:=\{E\mid E\subseteq E_A\text{ for some }A\in\mathcal{A}\}
  \] 
    where $E_A:=\{(x,y)\in G\times G\mid x^{-1}y\in A\}$.
\end{definition}

In particular, 
we can consider the coarse structure $\mathcal{E_{CB}}$ on $G$, associated to the ideal $\mathcal{CB}$ of coarsely bounded subsets of $G$. 

\begin{lemma}[{\cite[Cor.~2.23]{rosendal}}]\label{lem:Eob}
  For any topological group $G$, $\mathcal{E}_L = \mathcal{E}_{\mathcal{CB}}$.
\end{lemma}
\noindent
  Henceforth, we will always endow a topological group with
  the left coarse structure $\mathcal{E}_L =
\mathcal{E}_{\mathcal{CB}}$ and a pseudometric space $(X,d)$ with
the coarse structure $\mathcal{E}_d$.

\begin{definition}[{\cite[Defn.~2.43]{rosendal}}]\label{def:coarse_equivalence}
  Let $(X,\mathcal{E)}$ and $(Y,\mathcal{F})$ be coarse spaces.
  \begin{itemize} \item A map $\varphi:X\rightarrow Y$ is said to be
    \textit{bornologous} if
    $(\varphi\times\varphi)(\mathcal{E})\subseteq\mathcal{F}$

    \item A map $\varphi:X\rightarrow Y$ is said to be \textit{expanding} if
      $(\varphi\times\varphi)^{-1}(\mathcal{F})\subseteq\mathcal{E}$

    \item A map $\varphi:X\rightarrow Y$ is said to be a \textit{coarse
      embedding} if it is both bornologous and expanding.

    \item Let $Z$ be a set. Two maps $\alpha,\beta:Z\rightarrow X$
      are said to be \textit{close} if there exists $E\in\mathcal{E}$ such
      that $(\alpha(z),\beta(z))\in E$ for all $z\in Z$.

    \item A bornologous map $\varphi:X\rightarrow Y$ is said to be a
      \textit{coarse equivalence} if there exists a bornologous map
      $\psi:Y\rightarrow X$ such that  $\psi\circ\varphi$ is close
      to $\text{Id}_X$ and $\varphi\circ\psi$ is close to
      $\text{Id}_Y$.

    \item A subset $A\subseteq X$ is said to be \textit{cobounded} if there
      exists $E\in\mathcal{E}$ such that $$ X=E[A]:=\{x\in X\mid
      (x,y)\in E\text{ for some }y\in A\} $$

  \item A map $\varphi:X\rightarrow Y$ is \textit{cobounded} if $\varphi(X)$
is cobounded in Y. 
\end{itemize} 
\end{definition}

\begin{remark*}
  The terms in \cref{def:coarse_equivalence} agree with their
  usual (metric) definitions when $\mathcal{E}$ and $\mathcal{F}$ are
  \textit{metrizable} i.e. $\mathcal{E} = \mathcal{E}_d$
  and $\mathcal{F} = \mathcal{E}_{d'}$ for metrics $d,d'$ on
  $X,Y$ respectively. 
\end{remark*}

\begin{lemma}[{\cite[Lem. 2.45]{rosendal}}]\label{lem:ce_char}
  Any cobounded coarse embedding is a coarse equivalence.
\end{lemma}

\begin{lemma}\label{lem:cbaction}
  Suppose that a topological group $G$ acts continuously and isometrically 
  on a metric space $(X,d)$. Then the orbit map $\omega$ is
  bornologous.
\end{lemma}


\begin{proof}
  Since the action is continuous and isometric, the pullback metric
  $\omega^* d$ is a left-invariant continuous pseudometric on $G$
  and $(\omega \times \omega)[\mathcal{E}_{\omega^*d}] \subset
  \mathcal{E}_{d}$.  By definition $\mathcal{E}_L \subset
  \mathcal{E}_{\omega^*d}$, hence $\omega$ is bornologous.
\end{proof}

  Let $B_\alpha(x)$ denote the ball of radius $\alpha > 0$ centered
  at $x$. 

\begin{lemma}\label{lem:expanding_orbit_map}
    Suppose that a topological group $G$
    acts continuously and isometrically on a metric space $(X,d)$
    and let $\omega$ be the orbit map based at $x_0 \in X$.
    Then $\omega$ is expanding if $A_\alpha \dfn \omega^{-1}(B_\alpha(x_0))$
    is coarsely bounded for all $\alpha > 0$. 
\end{lemma}

\begin{proof}
    The orbit map $\omega$ is expanding if and 
    only if $(\omega\times\omega)^{-1}(\mathcal{E}_d)\subseteq\mathcal{E}_L$. 
    First consider $E_\alpha\in\mathcal{E}_d$ for $\alpha > 0$. Then 
    \begin{align*}
        (\omega\times\omega)^{-1}(E_\alpha)=&\ \{(g,h)\in G\times
        G\mid (g x_0,h x_0)\in E_\alpha\}\\ =&\
        \{(g,h)\in G\times G\mid d(g x_0,h
        x_0))<\alpha\}\\
        =&\ \{(g,h)\in G\times G\mid d(x_0,g^{-1}h x_0))<\alpha\}\\
        =&\ \{(g,h)\in G\times G\mid g^{-1}h\in A_\alpha\}\\
        =&\ E_{A_\alpha}
    \end{align*}

    Since $A_\alpha$ is coarsely bounded, $E_{A_\alpha} \in
    \mathcal{E}_{\mathcal{CB}} = \mathcal{E}_L$ by \cref{lem:Eob}.
    For general $E \in \mathcal{E}_d$, $E \subset E_\alpha$ for some
    $\alpha > 0$. Hence $(\omega \times \omega)^{-1}(E) \subset
    (\omega \times \omega)^{-1}(E_\alpha) \in \mathcal{E}_L$ and
    $(\omega \times \omega)^{-1}(E) \in \mathcal{E}_L$ as required.
  \end{proof}

We conclude by stating a convenient criterion for coarse
boundedness: 

\begin{proposition}[{\cite[Prop.~2.15(5)]{rosendal}}]
  \label{prop:ros_crit}
  Let $G$ be a Polish group.  A subset $A \subset G$ is coarsely
  bounded if and only if for every identity neighborhood $U \subset
  G$, there exists a finite set $F$ and $n \in \N$ such that $A
  \subset (FU)^n$.
\end{proposition}

\begin{corollary}\label{lem:fin_ext_cb}
    Let $G$ be a Polish group and $H\leq G$ be coarsely bounded in
    $G$. If $H\leq H'\leq G$ such that $[H:H']<\infty$ then $H'$ is
    also coarsely bounded in $G$.
\end{corollary}

\begin{proof}
    Since $H$ is coarsely bounded, for every open neighborhood $1\in
  U\subset G$, there exists a finite set $F$ and $n\in\N$ such that
$H\subset(FU)^n$. If $H'=\bigcup_{i=1}^k h_iH$, let
$F':=F\cup\{h_1,\ldots h_k\}$. Clearly $H'\subset(F'U)^n$ and hence
$H'$ is coarsely bounded in $G$. 
\end{proof}

\subsection{Local boundedness and bounded generation}\label{sec:cbgen}

Analogously to locally compact and compactly generated groups, we
introduce two classes of topological groups related to
the metrizability of $\mathcal{E}_L$.

\begin{definition} A topological group $G$ is
  \begin{enumerate}[label=(\roman*)]
      \item \textit{locally bounded} if there is a coarsely bounded
        neighborhood of identity.
    \item \textit{boundedly generated} if 
    it admits a coarsely bounded generating set.
    \end{enumerate}
\end{definition}
\begin{proposition}[{\cite[Thm.~2.40]{rosendal}}]\label{prop:gentoloc}
  Any boundedly generated Polish group is locally bounded.
\end{proposition}

\begin{remark*}\label{rem:lb_non_inheritence}
    The properties of local boundedness 
    and bounded generation are not inherited by 
    Polish subgroups. For example, consider the ladder
    surface $S$. Then $\Map(S)$ is boundedly generated 
    and hence locally bounded as well {\cite{mannrafi}} 
    but $\PMap(S)$ is neither locally bounded nor 
    boundedly generated {\cite{hill}}.
\end{remark*}

\begin{proposition}[{\cite[Cor.~3.26]{rosendal}}]
    Among Polish groups, the properties of being locally bounded  
    and boundedly generated are both invariant under coarse equivalence. 
    Moreover, every coarse equivalence between boundedly 
    generated Polish groups is automatically a quasi-isometry.
\end{proposition}

We recall that
$\mathcal{E}_L$ is metrizable when it is induced by
a  (possibly discontinuous) 
metric on $G$, in which case $\mathcal{E}_L$ defines a
 coarse-equivalence type for $G$ in the usual (metric) sense.  
 Crucially: 

\begin{theorem}[{\cite[Thm.~2.38]{rosendal}}]
Let $G$ be a Polish group.  Then $\mathcal{E}_L$ is metrizable if
and only if $G$ is locally bounded if and only if $\mathcal{E}_L$ is induced by a continuous left-invariant pseudometric $d$ on $G$.    
\end{theorem}

  When $\mathcal{E}_L$ is metrizable, by definition, it is maximal among the set of
coarse structures on $G$ induced by continuous left-invariant
pseudometrics, with
respect to the partial ordering 
\[\mathcal{E}_{d'} >
\mathcal{E}_{d} \iff \mathcal{E}_{d'} \subset \mathcal{E}_d \iff
\Id : (G,d') \to (G,d) \text{ is bornologous}\:.\]  
In particular, $\mathcal{E}_L$ is the unique 
such coarse structure.
Similarly, we may consider a (finer) partial ordering on the set of 
left-invariant
continuous pseudometrics on $G$: let $d \gg d'$ whenever $\Id :
(G,d) \to (G,d')$ is coarsely Lipschitz.  We observe that any 
maximal $d$ is unique up to quasi-isometry.

\begin{theorem}[{\cite[Prop.~2.72]{rosendal}}]\label{thm:max_gen}
 Let $G$ be a Polish group.  Then $G$ admits a continuous
 left-invariant pseudometric $d$ maximal
 with respect to $\ll$ if and only if $G$ is boundedly
 generated, if and only if $d$ is quasi-isometric to the word metric
 on $G$ with respect to a symmetric coarsely bounded generating set.
\end{theorem}

\noindent
It follows that the word metric on $G$ with respect to any coarsely bounded
generating set gives a well-defined quasi-isometry type whenever $G$
is boundedly generated. 


\subsubsection{Locally bounded subgroups of $\Map(S)$}\label{sec:lbdd_map} 


For a surface $S$, recall that 
\[\Map(S) \dfn \Homeo^+(S)/\Homeo_0(S)\] where
$\Homeo^+(S)$ is the group of orientation-preserving self-homeomorphisms
of $S$, endowed with the compact-open topology, and $\Homeo_0(S)$ is
its identity component.  The induced
(quotient) topology on $\Map(S)$ has a local (clopen) base at $\Id_S$
induced by the pointwise stabilizers $\tilde U_\Sigma \dfn \{f \in
  \Homeo^+(S) :
f|_\Sigma = \Id_\Sigma \}$ of compact,
essential subsurfaces $\Sigma \subset S$;  we
denote the elements of this local base
$U_\Sigma \dfn \tilde U_\Sigma / (\tilde U_\Sigma
\cap \Homeo_0(S)) < \Map(S)$.

\begin{remark*}\label{rmk:Gbase} 
    Let $G \leq \Map(S)$ a Polish subgroup. Given an essential
    compact subsurface $\Sigma \subset S$ let $\nu_\Sigma$ denote
    the (pointwise) $G$-stabilizer for $\Sigma$, that is
    $\nu_\Sigma \dfn U_\Sigma \cap G$.
    Since $\Map(S)$ has a local base $\{U_\Sigma\}_\Sigma$ 
    at $\Id_S$, likewise $G$ has a local base
    $\{\nu_\Sigma\}_\Sigma$ at $\Id_S$.
\end{remark*}

\begin{remark*}
    Since $G$ has a local base of open subgroups at $\Id_S$, $G$ is non-Archimedean.
\end{remark*}

The following is immediate from \cref{rmk:Gbase}:
\begin{lemma}\label{lem:lb_stab_subsurface}
    Let $G\leq\Map(S)$ be a locally
  bounded Polish subgroup. There exists a compact essential subsurface
$\Sigma\subset S$ whose 
stabilizer $\nu_\Sigma$ is  coarsely bounded in $G$. 
\end{lemma}

\subsubsection*{Some important subgroups}
Let $\Ends(S)$ denote the (Freudenthal) endspace of $S$ and $\Ends_g(S) \subset
\Ends(S)$ the subspace of non-planar ends.
By $\PMap(S) \leq \Map(S)$ we denote the \textit{pure mapping class
group of $S$}, which is the kernel of natural map 
\[\pi : \Map(S) \to \Homeo(\Ends(S),\Ends_g(S))\] 
obtained from the action of $\Map(S)$
on the endspace of $S$. Let $\PMap_c(S) \leq \PMap(S)$ 
denote the subgroup of compactly supported (necessarily pure) mapping classes.  
$\PMap(S)$ is closed in $\Map(S)$, hence it is a
Polish subgroup. $\PMap_c(S)$ is not closed when $S$ is
infinite-type; let
$\overline{\PMap_c(S)}$ denote its closure.

\begin{remark}
  When $\partial S = \varnothing$ we note that $\PMap_c(S) =
  \Map_c(S)$, the (more commonly studied) subgroup of compactly
  supported mapping classes.
\end{remark}

\subsection{Cayley--Abels--Rosendal graphs}\label{sec:car}

Analogous to Cayley--Abels graphs
for totally disconnected, locally compact groups,
Branman--Domat--Hoganson--Lyman \cite{bdhl}
define graphical models for boundedly generated Polish groups.   We
generalize these results to the locally bounded case in
\cref{sec:ccar}.

\begin{definition}[{\cite[\S3]{bdhl}}]\label{def:car}
  A connected, countable simplicial graph $\Gamma$ is a
  \textit{Cayley--Abels--Rosendal graph} for a topological group $G$
  if $G$ admits a continuous, vertex-transitive, cocompact, and simplicial action
  with coarsely bounded vertex stabilizers.
\end{definition}

\begin{proposition}[{\cite[Prop.~8]{bdhl}}]\label{prop:car}
  Let $G$ be a Polish group.  Then $G$ admits a
  Cayley--Abels--Rosendal graph if and only if $G$ is boundedly generated.
  Moreover, the orbit map of $G$ on any such graph is a
  quasi-isometry.
\end{proposition}

\subsection{Asymptotic dimension otherwise}\label{sec:asdim_ow}

We now shift our focus to the asymptotic dimension 
of mapping class groups. 
Asymptotic dimension was introduced by 
Gromov and gives a `large scale' notion of dimension; see
\cite{asdim_bell} for a survey of results.

\begin{definition}
    Let $X$ be a metric space. Then $\text{asdim}(X)\leq n$ if for
    every uniformly bounded open cover $\mathcal{U}$, there is a
    uniformly bounded open cover $\mathcal{V}$ of multiplicity $n+1$
    such that $\mathcal{U}$ refines $\mathcal{V}$. We say that
    $\text{asdim}(X)=n$ if $\text{asdim}(X)\leq n$ but 
    $\text{asdim}(X)\not\leq n-1$.
\end{definition}

\begin{proposition}[{\cite[Prop.~22]{asdim_bell}}]\label{prop:coarse_inv_asdim}
    Let $X$ and $Y$ be metric spaces with the standard coarse
    structure and $f:X\rightarrow Y$ a coarse embedding. 
    Then $\text{asdim}(X)\leq \text{asdim}(Y)$.
\end{proposition}

It follows that asymptotic dimension is a coarse
invariant and hence well-defined in the setting of locally bounded
Polish groups. In particular, we can look at the asymptotic
dimension of locally bounded surface mapping class groups. When $S$
is a finite-type surface, {\cite{bbf}} shows that the asymptotic
dimension of $\Map(S)$ is finite. In the case of infinite type
surfaces, the only result (as far as the authors know) appears in
{\cite{shift}}. We summarize the relevant details below.

Let $S$ be an infinite-type surface. 
Suppose that there exists a countable family of homeomorphic
subsurfaces $\Sigma_{i\in \Z} \subset S$, each with a single
boundary component, and a simple path 
$\gamma \subset S \setminus \bigcup_i
\mathring \Sigma_i$ intersecting each $\partial \Sigma_i$
sequentially and accumulating to two distinct ends. A \textit{shift
map} $\omega$ is
a homeomorphism supported on a regular neighborhood of
$\gamma \cup \left(\bigcup_i \Sigma_i\right)$, preserving $\gamma$ set-wise and
restricting to homeomorphisms $\Sigma_i \to \Sigma_{i+1}$.
If in addition $\braket{\omega}$ is not coarsely bounded in
$\Map(S)$, then it is an \textit{essential shift} \cite[\S1]{shift}.

\begin{theorem}[{\cite[Thm.~1.1]{shift}}]\label{thm:shift}
  If $S$ is stable and $\Map(S)$ is boundedly generated and 
  contains an essential shift, then $\asdim \Map(S) = \infty$.
\end{theorem}

When $S$ is stable, \cref{thm:asdim} and
\cref{thm:shift} fully classify the infinite asymptotic
dimension cases:

\begin{theorem}[{\cite[Thm.~1.6]{shift}}]\label{thm:class}
  Let $S$ be stable and $\Map(S)$ be boundedly generated. If
  $S$ contains neither a non-displaceable subsurface nor an
  essential shift, then $\Map(S)$ is coarsely bounded.
\end{theorem}

\subsection{Classification of local boundedness}\label{sec:class_mapS}
Since asymptotic dimension is well-defined 
for locally bounded mapping class 
groups, we recall results from \cite{mannrafi} and \cite{hill} that  classify the infinite type surfaces whose mapping class groups and pure mapping class groups are locally bounded. Here, for $A\subset\Ends(S)$, $M(A)$ is the set of maximal ends in $A$ with respect to the partial order on $\Ends(S)$ defined in \cite{mannrafi}. 

\begin{theorem}[{\cite[Thm. ~1.4]{mannrafi}}]
    Let $S$ be an infinite type surface. Then $\Map(S)$ is locally
    bounded if and only if there is a finite type surface
    $\Sigma\subset S$ such that the complimentary regions of $K$
    each have infinite type and zero or infinite genus, and
    partition $\Ends(S)$ into finitely many clopen sets
   \[ 
    \Ends(S)=\Big(\bigsqcup_{A\in\mathcal{A}}A\Big)\sqcup\Big(\bigsqcup_{P\in\mathcal{P}}P\Big)
  \] 
    such that:
    \begin{enumerate}
        \item Each $A\in\mathcal{A}$ is self-similar with
          $M(A)\subset M(\Ends(S))$ and
          $M(\Ends(S))\subset\bigsqcup_{A\in\mathcal{A}}M(A)$. \item
          each $P\in\mathcal{P}$ is homeomorphic to a clopen subset 
        of some $A\in\mathcal{A}$.
        
        \item for any $x_A \in M(A)$, and any neighborhood $V$ of
          the end $x_A \in S$, there is $f_V \in \Homeo(S)$ so that
          $f_V(V)$ contains the complimentary region to $K$ with end
        set $A$.
    \end{enumerate}
    Moreover, in this case $\nu_\Sigma$ is a coarsely bounded
    neighborhood of the identity.
\end{theorem}

\begin{theorem}[{\cite[Thm. ~1.1(b)]{hill}}]\label{thm:hill} 
    Let $S$ be an infinite type surface. 
    Then $\PMap(S)$ is locally bounded if and only if it is boundedly generated
    if and only if $|\Ends(S)|<\infty$ 
    and $S$ is not a Loch Ness monster 
    with (non-zero) punctures.
\end{theorem}

\begin{remark}
  The authors are unaware of any work concerning the local
boundedness of $\overline{\PMap_c(S)}$. 
\end{remark}

\section{Witness-cocompactness}\label{sec:ccpt}

We discuss \textit{cocompact} and \textit{witness-cocompact} arc and curve
models and sketch the proof of \cref{thm:ccpt_asdim}, which
we will use in \cref{sec:nondisp} to compute the asymptotic
dimension of certain locally bounded surface mapping class groups.
This section summarizes the results of
\cite{ga_asdim}, to which we direct the reader for full detail; it
is included here for convenience.

\subsection{Cocompact arc and curve models}\label{sec:ccpt_gm}

Let $S$ be a surface of arbitrary topological type and let 
$\mathcal{K}(S) \dfn K(V(\mathcal{AC}(S)))$ denote the set of finite collections of
simple arcs and curves on $S$.  Note the arcs and curves
in $u \in \mathcal{K}(S)$ need not be pairwise disjoint.  

\begin{definition}\label{def:ccpt}
  A \textit{(metric) arc and curve model for $G \leq \Map(S)$} is a connected 
  (metric) graph $\G$ with discrete $V(\G) \subset \mathcal{K}(S)$ 
  that admits an action of $G$ induced by the
  permutation of its vertices.  $\G$ is \textit{cocompact} if 
  this action is cocompact.
\end{definition}

\begin{remark*}
  Throughout \cref{sec:ccpt}, a (metric) arc and curve model \textit{on
  $S$} will mean a (metric) arc and curve model for some $G \geq \PMap_c(S)$.
\end{remark*}

\begin{remark*}\label{rmk:ccptint}
If $S$ is finite-type, then 
\begin{enumerate*}[label=\textit{(\roman*)}]
  \item $\PMap_c(S) = \PMap(S)$ and
  \item
  $\G$ is cocompact if and only if $i(u,u)$ and $i(u,v)$ are
  uniformly bounded for $u \in V(\G)$ and $(u,v) \in E(\G)$.
\end{enumerate*}
\end{remark*}

\begin{definition}\label{def:witness} 
  Let $\G$ be an arc and curve model on $S$.  
  A compact, essential ($\pi_1$-injective, non-peripheral) subsurface
  $W \subset S$ is a \textit{witness} for $\G$ if 
  $W$ does not contain a pants component 
  and every $u \in V(\G)$ intersects every component of $W$.
\end{definition}

\noindent
We note that witnesses are not assumed to be connected.  Let $\X^\G$
denote the set of witnesses of $\G$, and $\xc^\G \subset \X^G$ the
subset of connected witnesses.  A \textit{witness set} on $S$
is any collection of compact, essential subsurfaces without pants
components closed under enlargement and the action of
$\PMap_c(S)$. 

By \cite{hhs_gen}, the geometry of cocompact arc and curve models on
finite-type surfaces is well understood.  In particular: 

\begin{theorem}[{\cite[Thm.~4.12]{ga_asdim}}]\label{thm:hhs}
  Let $\G$ be a cocompact arc and curve model on a finite-type surface
  $\Sigma$.  Then
  $(\G,\X^\G)$ is an asymphoric hierarchically hyperbolic space with
  respect to subsurface projection to witness curve graphs 
  $\pi_W : \G \to 2^{\cc W}$, $W \in \X^\G$.
\end{theorem}

\noindent
The $\PMap(\Sigma)$-equivariant geometry of $\G$ is uniquely
determined by $\hat \X^\G$: 

\begin{theorem}[{\cite[Thm.~4.13]{ga_asdim}}]\label{thm:map_ccpt_wit}
  The map $\G \mapsto \xc^\G$ induces a bijection between
  equivariant quasi-isometry types of cocompact arc and curve models on
  $\Sigma$ and connected witness sets on $\Sigma$.
\end{theorem}

\begin{remark*}
  The above is functorial in the following
sense: whenever $\X^{\G'} \subset \X^\G$ (equivalently $\xc^{\G'} \subset
\xc^\G$), there is a canonical equivariant
coarsely surjective, coarse Lipschitz map $\iota: \G \to \G'$.
\end{remark*}

We note that \cref{thm:hhs} implies that cocompact $\G$ on a
finite-type surface is
$\delta$-hyperbolic if and 
only if it has no pair of disjoint, connected witnesses.  More
broadly, $\G$ admits a distance formula in the sense of
Masur--Minsky: there is some $K>0$ such that for any $u,v \in V(\G)$, 
\[
  d_\G(u,v) \approx \sum_{W \in \X^\G} [d_{\cc
  W}(\pi_W(a),\pi_W(b))]_K\:.
\]

\subsection{Subsurface projection}\label{sec:subproj}
Given a compact, essential, connected, non-pants subsurface
$\Sigma \subset S$, let $\mathcal{K}(S,\Sigma) \subset
\mathcal{K}(S)$ denote the
subset of collections containing an element that intersects $\Sigma$ essentially.
We construct a 
projection $\rho_\Sigma : \mathcal{K}(S,\Sigma) 
\to {\mathcal{K}(\Sigma)}$ as follows (see 
\cite[\S5.2]{curvenotes}). 
Let $\iota : \Sigma \hookrightarrow S$ be the inclusion map, 
let $p : S_\Sigma \to S$ be the 
covering space associated to $\pi_1(\Sigma) \cong \img \iota_* <
\pi_1(S)$, and let
let $\tilde \iota : \Sigma \hookrightarrow
S_\Sigma$ be the (unique) lift of $\iota$
into $S_\Sigma$.  Fix any homeomorphism $\sigma :
S_\Sigma \to \inter \Sigma \dfn \Sigma \setminus \partial
\Sigma$ that is a homotopy inverse for
$\tilde \iota|_{\inter \Sigma}$; 
note that $\sigma$ is unique up to homotopy, hence isotopy. Obtain 
$\tilde \sigma$ by composing $\sigma$ with the inclusion $\inter
\Sigma \hookrightarrow \Sigma$.
\[
\begin{tikzcd}
  & S_\Sigma \ar[d,"p"] 
  \ar[dl, bend right=35,"\tilde \sigma"'] \\ 
  \Sigma \rar[hook,"\iota"'] \ar[ur,hook,"\tilde \iota"] & S 
\end{tikzcd}
\]
Given $u \in \mathcal{K}(S,\Sigma)$,
let $\rho_\Sigma(u)$ 
be the closures of the 
non-peripheral components of $\tilde \sigma p^{-1}(u)$, up to
isotopy.  

One verifies that $\rho_\Sigma(u)$ 
is independent of the 
choice of representative for $\omega$ and $\sigma$.
  Likewise, $\rho_\Sigma$ is independent of the choice of embedding
  of $\Sigma$: if $\iota' : \Sigma \hookrightarrow
S$ is
isotopic to $\iota$, then the lift $\tilde\iota\,'$ is
isotopic to $\tilde \iota$ and thus $\sigma$ is likewise a homotopy
inverse for $\tilde \iota\,'|_{\inter \Sigma}$.

\begin{remark*}
  The definition here for $\rho_\Sigma$ differs slightly from that
  in \cite{ga_asdim}, which instead passes to the Gromov closure of
  $S_\Sigma$; however, the definitions are consistent.  We can
  likewise define $\rho_\Sigma(u)$ as the collection of essential
  intersections of $u$ with $\Sigma$.
\end{remark*}

The natural action of $\PMap(\Sigma)$ on $\mathcal{K}(\Sigma)$
defines an action of 
$\Map(\Sigma,\partial \Sigma)
\twoheadrightarrow \PMap(\Sigma)$.  Similarly, 
$\Map(\Sigma,\partial \Sigma) \curvearrowright
\mathcal{K}(S,\Sigma)$ via the homomorphism
$\Map(\Sigma,\partial \Sigma) \to \PMap_c(S)$ obtained by
extending by identity.

\begin{lemma}[{\cite[Lem.~4.14]{ga_asdim}}]\label{lem:projequi}
 $\rho_\Sigma: \mathcal{K}(S,\Sigma) \to \mathcal{K}(\Sigma)$
is $\Map(\Sigma,\partial \Sigma)$-equivariant.  
\end{lemma}

\begin{corollary}\label{cor:pbact}
  Let $\phi \in \PMap(\Sigma)$.  Then there exists $\psi
  \in \PMap_c(S)$ preserving $\mathcal{K}(S,\Sigma)$ 
  such that for any $\omega \in
  \mathcal{K}(S,\Sigma)$, $\phi\rho_\Sigma(\omega) =
  \rho_\Sigma(\psi \omega)$. 
\end{corollary}

\subsubsection{Witness-cocompactness}\label{sec:witness_gm}  Let $W
\subset S$ be a connected witness for an arc and curve
model $\G$ on $S$.  Note that $V(\G) \subset
\mathcal{K}(S, W)$ and obtain an arc and curve model $\G_W$ on $W$ as
follows: let $V(\G_W) = \rho_W(V(\G) \subset \mathcal{K}(W)$, and
obtain $E(\G_W)$ as the push-forward of the edge relation  on $\G$ by
$\rho_W$.  By \cref{cor:pbact} $\PMap(W)$ acts on $\G_W$ by
permuting its vertices and the map $\rho_W : \G \to \G_W$ is
$\Map(W,\partial W)$-equivariant; since $\G$ is connected, likewise
is $\G_W$. If $\G$ is a metric graph, then likewise push
forward the edge lengths on $\G$ to obtain a metric on $\G_W$; in
either case, $\rho_W$ is $1$-Lipschitz. 

\begin{definition}\label{def:witnessccpt}
  Let $\G$ be a connected (metric) arc and curve model on $S$.  Then $\G$ is
  \textit{witness-cocompact} if
  \begin{enumerate}
    \item $\G$ has a (compact) witness; and
    \item for every witness $W \subset S$, $\G_W$ is cocompact.
  \end{enumerate}
\end{definition}

\begin{remark*}\label{rmk:wccptint}
  From \cref{rmk:ccptint}, it follows that $\G$ is
  witness-cocompact if and only if $\G$ has a witness and for any
  witness $W$ there is a uniform bound on $i(\rho_W(u),\rho_W(u))$
  and $i(\rho_W(u),\rho_W(v))$ for $u \in V(\G)$ and $(u,v) \in
  E(\G)$.
\end{remark*}

\begin{lemma}\label{lem:eqisec}
  Let $\G$ be a witness-cocompact arc and curve model and let $W$ be a
  witness.  Then any $\Map(W,\partial W)$-equivariant section
  $\sigma_W : V(\G_W) \to V(\G)$ is a quasi-isometric embedding.
  \end{lemma}

\begin{proof}
  Since $\rho_W$ is Lipschitz, it suffices to show that $\sigma_W$
  is likewise Lipschitz.
  Since $\G_W$ is
  cocompact, it has finitely many orbits of edges $(\bar u,\bar v)
  \in E(\G_W)$. Since $\sigma_W$ is equivariant, there are likewise
  finitely many orbits of pairs $(\sigma_W(\bar u), \sigma_W(\bar
  v))$. Let $L$ be the maximum of the distances
  $d_\G(\sigma_W(\bar u),\sigma_W(\bar v))$ for $(\bar u,\bar v) \in
  E(\G_W)$. Then $\sigma_W$ is coarsely $L$-Lipschitz.
\end{proof}

\begin{remark}
  If $\G$ is witness-cocompact, then for each witness $W$, $\G_W$ is
  cocompact: up to quasi-isometry, we may endow
  $\G_W$ with the usual simplicial metric.
\end{remark}

\subsection{Asymptotic dimension lower bounds}\label{sec:asdim_ccpt}

We sketch the arguments from \cite{ga_asdim} to prove
\cref{thm:ccpt_asdim}.  We begin by computing lower bounds
for the asymptotic dimension of cocompact arc and curve models on
finite-type surfaces.

\subsubsection{For finite-type surfaces}\label{sec:asdim_ft}

Let $\Sigma$ be a finite-type surface with a cocompact arc and curve
model $\M$.  We aim to show the following:

\begin{theorem}[{\cite[Thm.~4.21]{ga_asdim}}]\label{thm:cm_hyp}
  Let $\Sigma$ be a genus $g$ finite-type surface.
  If $\M$ is a (non-empty) $\delta$-hyperbolic cocompact
   arc and curve model on $\Sigma$, then $\asdim \M \geq g -
  \lceil \frac 12 \chi(\Sigma) \rceil$. 
\end{theorem}

\begin{remark*}\label{rmk:flats}
  In the complementary case, when $\M$ is not $\delta$-hyperbolic or
equivalently when $\M$ has $\nu > 1$ disjoint connected
witnesses, it will suffice  that $\asdim \M \geq \nu$.  In
particular, $\nu$ is exactly the HHS rank of $(\M,\X^\M)$, which
bounds $\asdim \M$ from below \cite[Thm.~1.15]{quasiflats}.
\end{remark*}

We prove \cref{thm:cm_hyp} by finding a compact subspace $Z \subset
\partial \M$ of known topological dimension.  For proper
$\delta$-hyperbolic spaces, the topological dimension of the
boundary gives bounds on the asymptotic dimension of the space
\cite[Prop.~6.2]{bleb}; while $\M$ is typically non-proper, a minor adaptation of
the lower bound suffices. 

\begin{proposition}[{\cite[Prop.~2.5]{ga_asdim}}]\label{prop:bdry_cpt}
  Let $X$ be a geodesic $\delta$-hyperbolic space with $Z \subset
  \partial X$ compact.  Then $\asdim X \geq \dim Z + 1$.
\end{proposition}

We find $Z$ as follows.  Recall that, whenever $\M$ and $\M'$ are cocompact
graph models on $\Sigma$ and $\X^{\M} \supset
\X^{\M'}$, there is a canonical coarsely surjective, coarsely
Lipschitz map $\iota : \M \to \M'$.  In particular, since $\X^{\cc
\Sigma} = \{\Sigma\}$, such a map $\iota : \M \to \cc \Sigma$ 
exists for any cocompact graph model $\M$.  We first prove that
when $\M$ is $\delta$-hyperbolic 
these maps are coarsely alignment preserving in the sense
of Dowdall--Taylor \cite{align}:  there exists $K$ for $\M,\M'$ such that 
for any aligned triple of vertices $(x,y,z) \in
V(\M)^3$, $d(\iota(x),\iota(y)) + d(\iota(y),\iota(z)) \leq
d(\iota(x),\iota(z)) + K$.

\begin{lemma}[{\cite[Lem.~4.22]{ga_asdim}}]\label{lem:canon_align}
  Let $\M$ and $\M'$ be arc and curve models on a finite-type surface
  such that $\X^\M \supset \X^{\M'}$, and let $\iota : \M
  \to \M'$ be the canonical coarse surjection.  If $\M$ is
  $\delta$-hyperbolic, then $\iota$ is
  coarsely alignment-preserving.
\end{lemma}

\noindent
The proof of \cref{lem:canon_align} follows from the existence of
hierarchy paths \cite[Thm.~4.4]{hhsii} in $\M,\M'$. Such paths are close to
geodesics and have projections
to witness curve graphs that are unparameterized quasi-geodesics.
Applying the distance formulas for $\M,\M'$ derives the claim.

Crucially, the theory of alignment preserving maps implies an
embedding of $\partial \M'$ into $\partial \M$, and in particular an
embedding $\partial \cc \Sigma \hookrightarrow \partial \M$ whenever $\M$ is
$\delta$-hyperbolic.



\begin{theorem}[{\cite[Thm.~3.2]{align}}]
  \label{thm:dtboundaries}
  Let $f : X \to Y$ be a coarsely surjective, coarsely alignment
  preserving map between geodesic 
  $\delta$-hyperbolic spaces.  Then $f$ induces an embedding
  $\partial Y \hookrightarrow \partial X$.
\end{theorem}

\noindent
When $\Sigma$ is a punctured sphere, we then conclude using a result
of Gabai:
\begin{theorem}[{\cite[Thm.~1.2]{gabai}}]\label{thm:gabai}
   Let $\Delta$ be the $(n + 4)$-times punctured sphere for $n \geq 0$.  
   Then $\partial \cc \Delta$ is
   homeomorphic to the $n$-dimensional N\"obeling space
   $\R^{2n+1}_n$.
\end{theorem}
\noindent
In particular, by the universal embedding property of N\"obeling
spaces \cite{nobeling} any $n$-dimensional compactum $Z$ embeds into $\partial
\cc \Delta \subset \partial \M$.  For general $\Sigma$, we apply a
result of Rafi--Schleimer \cite[Thm.~7.1]{rafischleimer} to obtain an
embedding of $\partial \cc \Delta$ into $\partial \cc \Sigma \subset
\partial \M$, which
completes the proof of \cref{thm:cm_hyp}.

\begin{proposition}[{\cite[Prop.~4.23]{ga_asdim}}]\label{prop:elSembed}
   Let $\Sigma$ be a finite-type hyperbolic surface of genus $g$ 
   and $\Delta$ the $(n+4)$-times punctured sphere, where $n = g - 1 -
  \lceil \frac 12 \chi(\Sigma) \rceil$. Then
   $\partial \cc \Delta$ embeds into $\partial \cc \Sigma$.
 \end{proposition}

\subsubsection{For infinite-type surfaces}\label{sec:asdim_it}

We prove the following:

\begin{theorem}\label{thm:ccpt_met_asdim}
  Let $S$ be an infinite-type surface and let $\M$ be
  a witness-cocompact metric arc and curve model on $S$. Then
  $\asdim V(\mathcal{M}) = \infty$.
\end{theorem}
\noindent
In particular, since $V(\M)$ is given the induced metric, it
isometrically embeds into $\M$ and \cref{thm:ccpt_asdim} follows.
By \cref{lem:eqisec} and the monotonicity of asymptotic dimension, 
it suffices to find for every $d \in \N$ some witness $W \subset S$
for which $\asdim \M_W \geq d$ (see\ \cite[\S 4.3.2]{ga_asdim}). 

Given a witness-cocompact arc and curve model $\M$ on an infinite-type surface
$\Omega$, let $w_\M \in \N\cup \{\infty\}$ denote the least upper
bound on cardinalities for a set of pairwise-disjoint 
connected witnesses for $\M$. If $w_\M$ is infinite, then for each
$d$ fix a compact subsurface $\Sigma_d$ containing at least $d$
pair-wise disjoint connected witnesses for $\M$. These witnesses are
likewise witnesses for $\M_{\Sigma_d}$, hence $\M_{\Sigma_d}$ has
rank $\nu \geq d$ and $\asdim \M_{\Sigma_d} \geq d$.  If $w_\M$ is
finite, then fix a set $\{\Delta_i\}$ of $w_\M$ pairwise disjoint witnesses, 
with $\Delta_0$ a witness adjacent to an infinite-type component of
$S \setminus \bigcup_i \Delta_i$.  By enlarging $\Delta_0$
disjointly from the remaining $\Delta_i$, we obtain
compact subsurfaces $\Sigma_d \subset S \setminus \bigcup_{i > 0}
\Delta_i$ such that $-\chi(\Sigma_d) > 2d$ and $\Delta_0 \subset
\Sigma_d$.  Since $\Sigma_d \supset \Delta_0$, it is a witness for
$\M$, and each $\M_{\Sigma_d}$ must have rank $\nu = 1$ else we
obtain a set of witnesses for $\M$ of cardinality greater than $w_\M$. It
follows that $\M_{\Sigma_d}$ is $\delta$-hyperbolic.  Applying
\cref{thm:cm_hyp}, we obtain $\asdim \M_{\Sigma_d} \geq d$ as
required. \cref{thm:ccpt_met_asdim} follows. 
\sqed

\section{A \v Svarc--Milnor lemma for locally bounded
groups}\label{sec:ccar}

\begin{definition}\label{def:bccpt} 
  The action of a group $G$ on a metric graph $\Gamma$ is \textit{bounded-cocompact}
  if, for every closed bounded subgraph $\Lambda \subset \Gamma$, $G\Lambda/G$ is compact.
\end{definition}

\begin{definition}\label{def:ccar}
 A connected metric graph $\Gamma$ with a discrete vertex set $V(\Gamma)$ along with an
 isometric, isomorphic, and continuous action of a group $G$ is a 
 \textit{coarse Cayley--Abels--Rosendal graph for $G$}
  if the action is vertex-transitive and
 bounded-cocompact with coarsely bounded vertex stabilizers.
\end{definition}


Recall that a Polish group $G$ is \textit{non-Archimedean} if it has a
(clopen) subgroup neighborhood basis at identity \cite{kechris}.  The following
extends \cref{prop:car}.

\begin{proposition}\label{prop:ccar}
 Let $G$ be a Polish group. If $G$ 
 admits a coarse Cayley--Abels--Rosendal graph then it is locally
 bounded; moreover, for
 any such graph $\Gamma$ the orbit map to 
 $V(\Gamma)$, with the induced metric,
 is a coarse equivalence.  
 If $G$ is non-Archimedean then the converse holds.
\end{proposition}

\begin{proof}
  Let $\omega : G \to V(\Gamma)$ denote the vertex orbit map; since
  the action is continuous, $\omega$ is continuous.  
  Moreover, since $V(\Gamma)$ is discrete, the stabilizer of a vertex in $\Gamma$ is thus a coarsely
  bounded neighborhood of identity and $G$ is locally bounded.  We
  must show that $\omega$ is bornologous, expanding, and cobounded,
  hence a coarse equivalence: 
  the first follows from \cref{lem:cbaction} and the last from
  vertex transitivity, hence we need only check that $\omega$ is
  expanding. 

  Let $d$ denote the metric on $\Gamma$ and fix a vertex $x \in
  V(\Gamma)$; we assume $\omega$ is the orbit map based at $x$. 
  By \cref{lem:expanding_orbit_map}, 
  it suffices to show that $A_\alpha = \omega^{-1}(B_\alpha(x))$ is coarsely
  bounded.     Fix a connected bounded subgraph $\Lambda_\alpha
  \subset \Gamma$ containing $B_\alpha(x) \cap \img \omega =
  B_{\alpha}(x) \cap V(\Gamma)$ and let $A'_\alpha =
  \omega^{-1}(\Lambda_\alpha) = \{g \in G : gx \in
  V(\Lambda_\alpha)\}$.  Since $V(\Gamma)$ is discrete and $G$ acts
  vertex-transitively, the infimum of edge lengths in $\Gamma$ is
  non-zero. Hence by bounded-cocompactness, $\Lambda_\alpha / G$ is
  a finite graph, 
  or equivalently $\Lambda_\alpha$ intersects finitely many $G$-orbits of
  edges: the midpoints of edges in $\Lambda_\alpha/G$ are discrete, hence
  must be finite by compactness.   Let $\nu_x \dfn \stab_G(x) \leq G$ denote the
  stabilizer of $x$ and fix a finite set of elements $F_\alpha
  \subset G$ so
  that if $(gx,hx) \in E(\Lambda_\alpha)$, then $g^{-1}h \in
  \nu_x F_\alpha \nu_x$;
  additionally add some element $g_0 \in A'_\alpha$.  Let $m = \diam
  \Lambda_\alpha$.  Then $A'_\alpha \subset (F_\alpha \nu_x)^m$,
  hence $A'_\alpha \supset A_\alpha$ 
  is coarsely bounded since likewise is $\nu_x$.

  The converse when $G$ has small subgroups is shown in the
  following lemma.
\end{proof}

\begin{lemma}\label{lem:ccar_existence}
  If $G$ is a non-Archimedean locally bounded Polish group, 
  then it admits a coarse Cayley--Abels--Rosendal graph.
\end{lemma}

\begin{proof}
  Fix a coarsely bounded clopen subgroup $H \leq G$ and a countable
  set $Z = \{z_i\}\subset G$ such that $G = \braket{Z,H}$.
  Such a $Z$ always exists: for example, since $G$ is separable
  and $H$ open, $H$ has a countable
  transversal in $G$. 
  Construct a metric graph $\Gamma$
  on the vertex set $G / H$ by attaching an edge of length $i$
  between $g H$ and $k H$ whenever $g^{-1}k \in Hz_i H$. The set
  $Z$ generates $G$ over $H$, hence $\Gamma$ is connected.
  Since $H$ is clopen, the left action of $G$ on 
  $G/H$ is continuous; 
  it induces a continuous, isometric, isomorphic, and
  vertex transitive action on $\Gamma$ with coarsely bounded vertex
  stabilizer $\stab_G(H) = H$.

  We verify bounded-cocompactness: $\Gamma / G$ is the metric
  graph isomorphic to a bouquet of countably many circles $e_i$, each of
  length $i$.  It suffices that 
  if $\Lambda$ is a subgraph of $\Gamma$ for which $\Lambda/G$ is
  not compact, then it is unbounded.  In particular, 
  $\Lambda/G \subset \Gamma / G$ must contain 
  infinitely many edges and thus an edge of length at least
  $i$ for every $i \in \N$, hence likewise does $\Lambda$.  
\end{proof}

\begin{remark*}\label{rmk:bdgen_finiteZ}
    If there exists a finite subset $Z\subset G$ and a coarsely 
    bounded clopen subgroup $H$ such that 
    $G=\langle Z ,H\rangle$, then clearly $H\cup Z$ is 
    also coarsely bounded and hence $G$ is boundedly generated. 
    Conversely, if $G$ is boundedly generated, then we may choose $Z$ 
    in \cref{lem:ccar_existence} above 
    to be finite by \cref{prop:ros_crit}. 
    Hence $G$ is boundedly generated if and only if $Z$ can be chosen to be finite.
\end{remark*}

\begin{remark*}\label{rmk:ccar_build}
  For a Polish group $G$, let $H \leq G$ be a coarsely bounded
  open subgroup and $Z \subset G$ an enumerated countable set that generates $G$
  over $H$.  Let $\cc_{H,Z}(G)$ denote the coarse
  Cayley--Abels--Rosendal graph constructed as in the proof of
  \cref{lem:ccar_existence}; by \cref{prop:ccar}, its
  vertex set is coarse equivalent to $G$.
\end{remark*}

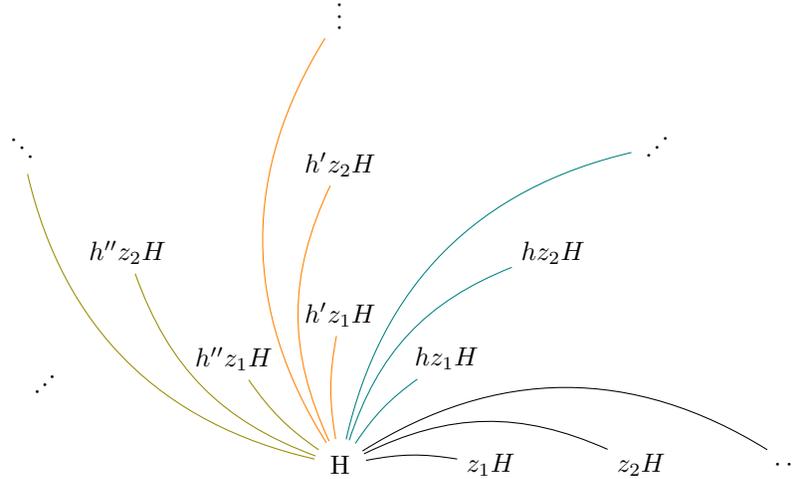
\begin{figure}[h]
  \centering
  \begin{tikzpicture}
    \node[circle] (H) at (0,0) {H};
    \node[rotate=135] at (-4,1) {\vdots};
    \foreach \h / \angle / \color in { /0/black, h/45/teal,
    h'/90/orange, h''/135/olive}
      {
        \begin{scope}[rotate=\angle]
          \node (\h 1) at (2,0) {$\h z_1 H$};
          \node (\h 2) at (4,0) {$\h z_2 H$};
          \node[rotate=\angle] (\h 3) at (6,0) {$\cdots$};
          \draw (H) edge[bend left=10,color=\color] (\h 1);
          \draw (H) edge[bend left=25,color=\color] (\h 2);
          \draw (H) edge[bend left=32,color=\color] (\h 3);
        \end{scope}
      }
  \end{tikzpicture}
  \caption{The neighborhood of the vertex $H$ in $\cc_{H,Z}(G)$.
  Here, $h,h',h'' \in H$ and $Z = \{z_i\}$}
  \label{fig:chz}
\end{figure}

\begin{remark*}\label{rmk:car_build}
  The construction of $\cc_{H,Z}(G)$ exactly coincides
with that in \cite[\S3]{bdhl} when $Z$ is finite 
and generates $G$ over $H$.   In this case,
  $\cc_{H,Z}(G)$ has only finitely many edge orbits
  and $G$ acts \textit{cocompactly} on $\cc_{H,Z}(G)$, hence
  $\cc_{H,Z}(G)$ (viewed as a simplicial graph) is a
  Cayley--Abels--Rosendal graph for $G$.  
\end{remark*}

\section{Arc and curve models for subgroups of the mapping class
group }\label{sec:model_general}

In this section, we adapt the construction in \cref{lem:ccar_existence} 
to the context of subgroups of a mapping class group of an infinite type surface $S$. 
More specifically, if 
$G\leq\Map(S)$ is a locally bounded Polish subgroup, we 
construct an arc and curve model (\cref{def:ccpt}) that 
is also a coarse Cayley-Abels-Rosendal graph for $G$.

Let $S$ be an infinite-type surface 
and let $G\leq \Map(S)$ be a locally bounded Polish subgroup.
Suppose $\mu\in\mathcal{K}(S)$ is a finite collection of simple arcs and simple closed 
curves with a coarsely bounded (set-wise) 
$G$-stabilizer $\nu_\mu \dfn \stab_G(\mu)$. Then there exists a countable set
$Z\subset G$ such that $G=\langle\nu_\mu, Z\rangle$.
Define a graph $\mathcal{M}_{\mu,Z}(G)$ with vertex 
set $V(\mathcal{M}_{{\mu,Z}}(G)):=G\mu$. Consider the $G$-equivariant bijection 
$V(\mathcal{C}_{\nu_\mu,Z}(G))=
G/\nu_\mu \xrightarrow{\simeq}G\mu=V(\mathcal{M}_{\mu,Z}(G))$
defined by $g\nu_\mu\mapsto g\mu$ and obtain $\mathcal{M}_{\mu,Z}(G)$ by
pushing 
forward the (metric) edge relation in $\cc_{\nu_\mu,Z}(G)$ to
$V(\mathcal{M}_{\mu,Z}(G))$. Note that $\mathcal{M}_{\mu,Z}(G)$ is 
$G$-equivariantly isometric to $\cc_{\nu_\mu,Z}(G)$ and hence a 
coarse Cayley--Abels--Rosendal graph for $G$: if there exists such a $\mu$, then 
$G$ is coarsely equivalent to $\mathcal{M}_{\mu,Z}(G)$. 

By \cref{rmk:bdgen_finiteZ}, $G$ is boundedly generated
if and only if $Z$ can be chosen to be finite, which by \cref{rmk:car_build}
occurs if and only if $\cc_{\nu_\mu,Z}(G)$ 
(and hence $\mathcal{M}_{\mu,Z}(G)$) is a Cayley--Abels--Rosendal graph. 
The following lemma shows that there indeed exists a $\mu$ with 
coarsely bounded $G$-stabilizer $\nu_\mu$, whence \cref{thm:model}
follows.

\begin{lemma}\label{lem:mu_exists}
  Let $G\leq\Map(S)$ be a locally bounded Polish 
  subgroup. Then there
  exists $\mu\in\mathcal{K}(S)$ such 
  that $\nu_\mu$ 
  is coarsely bounded in $G$. 
\end{lemma}

\begin{proof}
    Since $G$ is locally bounded, \cref{lem:lb_stab_subsurface} 
    tells us that there exists $\Sigma\subset S$ 
    with a coarsely bounded $G$-stabilizer $\nu_\Sigma$.
    Let $\mu_0\in\mathcal{K}(S)\cap\mathcal{K}(\Sigma)$ be a filling 
    collection of arcs and curves in $\Sigma$ and 
    $\mu:=\mu_0\cup\partial\Sigma$; clearly $\nu_\Sigma\leq\nu_\mu$. Let $S(\mu)$ 
    denote the permutation group on $\mu$ and let $K$ denote the kernel 
    of the natural map $\nu_{\mu}\rightarrow S(\mu)$. 
    In general 
    $\nu_\Sigma\subseteq K$. Since $\mu_0$ is filling, 
    by Alexander's method {\cite[Prop. ~2.8]{fm}} we have in fact 
    $\nu_\Sigma=K$. Finally since $\mu$ is a finite collection, 
    so is $S(\mu)$ and consequently $[\nu_\Sigma:\nu_\mu]\leq|S(\mu)|<\infty$.
    By \cref{lem:fin_ext_cb}, $\nu_\mu$  
    is also coarsely bounded in $G$, as required. 
\end{proof}

\section{Asymptotic dimension}\label{sec:nondisp}

Throughout this section, we assume $S$ is an infinite-type
surface with a non-displaceable subsurface for a 
Polish subgroup $\PMap_c(S) \leq G \leq \Map(S)$ such that $G$ is boundedly
generated or  $G \in
\{\overline{\PMap_c(S)},
\PMap(S), \Map(S)\}$ and locally bounded.  
To prove
\cref{thm:asdim}, it suffices to choose $\mu$ and $Z$ so that
 $\M = \M_{\mu,Z}$ 
 is witness-cocompact: by
\cref{thm:ccpt_met_asdim} $\asdim V(\M) = \infty$, hence likewise $\asdim
\Map(S) = \infty$ by \cref{thm:model}(1).  Since $\PMap_c \leq G$ it suffices to 
ensure that 
\begin{enumerate}[label=\textit{(\roman*)}]
  \item
  $\M$ has a witness, and (by \cref{rmk:wccptint}) 
\item 
that edge- and self-intersection numbers are uniformly bounded.
\end{enumerate}

Only condition \textit{(i)} uses non-displaceability; note also that it may be
satisfied for arbitrary locally bounded Polish subgroups $G \leq
\Map(S)$.  
We choose $\mu$ so that $\nu_\mu$ is coarsely bounded and
$V(\M) = G  \mu$
has a witness. Let
$\Delta \subset S$ to be a compact, essential, $G$-non-displaceable
subsurface sufficiently large that $\nu_\Delta$ is coarsely bounded.
Fix $\mu_0 \in \mathcal{K}(\Delta) \cap \mathcal{K}(S)$ to be a filling
collection of curves in $\Delta$ and let $\mu = \mu_0 \cup \partial
\Delta$.  By \cref{lem:mu_exists},
$\nu_\mu$ is coarsely bounded. For any $g \in
G$, $g \mu_0$ is
filling in $g \Delta$, hence since $g \Delta$ intersects
$\Delta$ essentially likewise does $g \mu = g\mu_0 \cup \partial
g\Delta$: $\Delta$ is a
witness for $\M$.

Intersection numbers are invariant under $G \leq \Map(S)$, hence to
ensure \textit{(ii)} it suffices that
\begin{itemize}
  \item[\textit{(ii')}]
$i(\mu, z_i\mu)$ is uniformly bounded over $z_i \in Z$.
\end{itemize}
In particular, $i(g
\mu,g \mu) = i(\mu,\mu) <\infty$ 
and if $(g\mu,k\mu) \in E(\M)$ then $g^{-1}k =  hz_ih'$ for
some $z_i \in Z$ and $h,h' \in \nu_\mu$, thus $i(g\mu,k\mu) =
i(h^{-1}\mu,z_ih'\mu) = i(\mu,z_i\mu)$.  
When $G$ is
boundedly generated we may choose $Z$ to be finite, hence 
\textit{(ii')} is immediate.
For the remainder of the section, we construct for each locally bounded case 
a (countable) $Z\subset G$
 that generates $G$ over $\nu_\mu$ and
satisfies \textit{(ii')}.


\subsection{Enforcing small
intersection}\label{sec:sm_int_Z}

We first produce topological generating sets for
$G = \overline{\PMap_c(S)},\PMap(S)$ satisfying \textit{(ii')}.  In
particular, since $\nu_\mu$ is open these generate $G$ over $\nu_\mu$.

\begin{lemma}
    There exists a countable generating set $T$ for $\PMap_c(S)$ 
    such that $\{i(\mu,t\mu)\}_{t\in T}$ is 
    finite and hence bounded above.
\end{lemma}

\begin{proof}
    Let $\Sigma_0\subset\Sigma_1\subset\Sigma_1\subset\cdots$ be a compact exhaustion of $S$ such that 
    \begin{enumerate}
        \item $\Sigma_0\supset\Delta$.
        \item If $C_j^i$ denotes the simple closed curves corresponding to the Dehn-Likorish generators for $\Map(\Sigma_i)$, then $\{C_j^i\}_j\cup\partial\Sigma_i\subset \{C^{i+1}_k\}_k\cup\partial\Sigma_{i+1}$.
    \end{enumerate}
    Then the collection of Dehn twists 
    $T:=\{T_{ij}\}$ along the simple 
    closed curves $C_j^i$ generate 
    $\PMap_c(S)$. Moreover, for 
    sufficiently large 
    $i$, $\{C_k^{i+1}\}\setminus\left(\{C_j^i\}\cup\partial\Sigma_i\right)$ 
    only consists of simple closed 
    curves outside of $\Delta$ and hence their corresponding Dehn twists fix 
    $\Delta$ pointwise and 
    therefore $\mu$. As such, only finitely many of 
    the Dehn twists $T_{ij}$ act 
    non-trivially on $\mu$, 
    hence $\{i(\mu,t\mu)\}$ is a finite collection bounded 
    above by the maximum over these finitely many Dehn twists 
    $t\in T$ that act non-trivially on $\mu$.
\end{proof}

Recall that a \textit{handle shift} is a shift map (see
\cref{sec:asdim_ow}) with homeomorphic
subsurfaces $\Sigma_i \cong \Sigma_{1}^1$, and let $h_\pm \in
\Ends_g(S)$ denote the (forward and backward) 
accumulation points of the underlying path, which we will call the
\textit{endpoints} of $h$.  

\begin{lemma}\label{lem:pmap_gen}
  Let $H \subset \PMap(S)$ be a collection of handle shifts 
  such that $\{(h_-,h_+) : h \in H\}$ is dense in
  $\Ends_g(S)\times \Ends_g(S)$.  For any neighborhood
  $1 \in \nu \subset \Map(S)$, $\PMap(S) \leq
  \braket{H,\PMap_c(S),\nu}$.
\end{lemma}

\begin{proof}
    Let $\nu_P:=\nu\cap\PMap(S)$ be a clopen subgroup of $\PMap(S)$.      
    We know that 
    $\PMap(S)$ is topologically 
    generated by Dehn twists (which are compactly supported)
    and handle shifts {\cite[Thm.~4]{pv}}, {\cite[Cor.~6 and Section 
    2.3]{apv}}. Since $H$ is dense in 
    $\Ends_g(S)\times \Ends_g(S)$, $\PMap(S)$ is in fact topologically 
    generated by Dehn twists and $H$ 
    {\cite[Thm.~4.4]{avbig}}. If we consider translates of $\nu_P$ by 
    Dehn twists and elements of $H$, we therefore  get an open cover of 
    $\PMap(S)$. Hence $\PMap(S)=\langle H,\PMap_c(S),\nu_P\rangle$ which 
    implies $\PMap(S)\leq\langle H,\PMap_c(S),\nu\rangle$.
\end{proof}

\begin{lemma}\label{lem:hshift}
  There exists a countable set of handle shifts $H \subset \PMap(S)$
  whose endpoints are dense in $\Ends_g(S) \times
  \Ends_g(S)$ and for which $i(\mu,h\mu)$ is uniformly bounded
  for $h \in H$.
\end{lemma}

\begin{proof}
  Let $S_1,\ldots S_b$ be the complementary components of $\Delta$. 
  Fix $k=\binom{b}{2}$ many pairwise 
  disjoint strips $s_{\{i,j\}}$, 
  connecting the $i^\text{th}$ and 
  $j^\text{th}$ complementary components
  of $\Delta$ with $i\neq j$. For two 
  distinct ends $x,y \in \Ends_g(S)$ accumulated by 
  genus, consider handle shifts $h_{xy}$ such that  
  \begin{enumerate}
    \item If $x,y\in S_i$, then $\text{Domain}(h_{xy})\cap\Delta=\varnothing$. 
    \item If $x,x'\in S_i,\ y,y'\in S_j$ and $i\neq j$, 
    then $\text{Domain}(h_{xy})\cap\Delta=s_{\{i,j\}}$ and 
    $h_{xy}|_{s_{\{i,j\}}}=h_{x'y'}|_{s_{\{i,j\}}}$.
    \item For $x,y\in\Ends_g(S)$, $h_{yx}=h_{xy}^{-1}$.
  \end{enumerate}
  
  Fix $E\subset\Ends_g(S)\times\Ends_g(S)$ 
  a countable dense subset and set $H:=\{h_{xy}\mid(x,y)\in 
  E\}$.
  Let $h_{ij}:=h|_{s_{\{i,j\}}}$ for 
  $h\in H$ such that $h_-\in S_i,\ h_+\in S_j$. 
  Note that by $(2)$, this is well-defined 
  and for any $h\in H$, since $\mu\subset\Delta$,
  either $h$ fixes $\mu$ pointwise or 
  $i(\mu,h\mu)=i(\mu,h_{ij}\mu)$ for some $i\neq j$. 
  Hence $i(\mu,h\mu)$ is uniformly bounded 
  by $\max_{i,j} i(\mu,h_{ij}\mu)$ for any $h\in H$ as required.
\end{proof}

\noindent
When $G = \overline{\PMap_c(S)}$, it suffices that $Z = T$,
and when $G = \PMap(S)$, that $Z = T \cup H$.  \sqed

\begin{remark*}
By \cite{hill} (\cref{thm:hill}) $G = \PMap(S)$ is locally bounded if and only if it
is boundedly generated, in which case $Z$ may be chosen to be
(in fact) finite.  We include \cref{lem:hshift} so that our
arguments are self-contained.
\end{remark*}

\begin{figure}[h]
  \centering
  \begin{tikzpicture}
    \node[line width=1pt] (surface) at (0,0)
    {\includegraphics[width=0.5\textwidth]{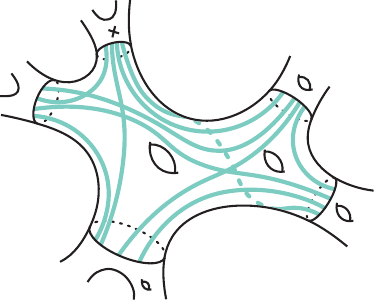}};
  \node at (-2.2,-1.5) {$S_1$};
  \node at (-2,2) {$S_2$};
  \node at (-.65,-.6) {$s_{\{1,2\}}$};
  \node at (1.2,-1.35) {\LARGE $\Delta$};
  \end{tikzpicture}
  \caption{The strips $s_{\{i,j\}}$ in $\Delta$}
  \label{fig:strips}
\end{figure}
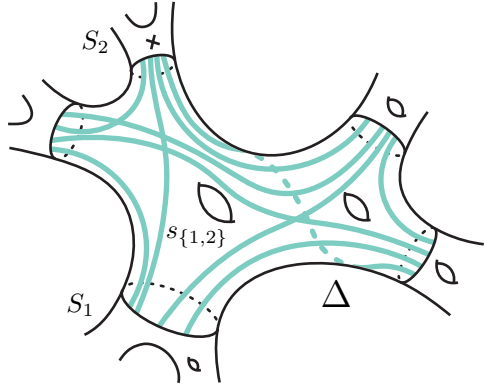

We conclude with the case when $G = \Map(S)$. We aim to construct a
transversal 
$I$ for $K \dfn \braket{T,H,\nu_\mu}$ that satisfies
the intersection condition \textit{(ii')}; since $K$ is open, $I$ is countable and we may set
$Z = T \cup H \cup I$. Let $E \dfn\Ends(S)$ and $E_g \dfn \Ends_g(S)$
and consider the exact sequence
\[
1\rightarrow\PMap(S)\rightarrow\Map(S)\xrightarrow{\pi}\Homeo(E,E_g)
\rightarrow 1 \:.
\]
It suffices to construct a (set-theoretic) section $\sigma$ for $\pi$ whose
whose image satisfies \textit{(ii')}.  Then $I_\sigma \dfn \img
\sigma$ is a transversal for $\PMap(S)$ (albeit possibly
uncountable), and since $\PMap(S) \leq K$
by \cref{lem:pmap_gen}, $I_\sigma$ contains a transversal $I \subset
I_\sigma$ for $K$ likewise satisfying \textit{(ii')}.
We construct $\sigma$ below.


\begin{lemma}\label{lem:end_section}
  There exists a (set-theoretic) section $\sigma :
  \Homeo(E,E_g) \to
  \Map(S)$ such that $i(\mu,\sigma(\varphi)\mu)$ is uniformly
  bounded over $\varphi \in \Homeo(E,E_g)$.
\end{lemma}

\begin{remark}
  In the following, let $\Sigma^b_{g,p}$ denote the orientable
  surface with genus $g$, $b$ boundary components, and $p$
  punctures.
\end{remark}

\begin{proof}
  Fix some connected, essential subsurface $\Pi \cong
  \Sigma^b_{g,0} \subset S$ such that $\Pi \supset
  \Delta$ and whose complementary components have either zero or
  infinite genus.
  Fix an embedding of $\Pi$
  into $\Sigma = \Sigma_{g,0}^{b^2}$ such that $\Sigma
  \setminus \Pi$ is the disjoint union of copies of
  $\Sigma^b_{0,1}$.  Let $\sigma_0 :
  \operatorname{Sym}(\pi_0(\partial \Sigma)) \to \Map(\Sigma)$ be a
  choice of (set-theoretic) section; note that $i(\mu,\sigma_0(\alpha)\mu)$ is
  uniformly bounded over $\alpha \in
  \operatorname{Sym}(\pi_0(\partial \Sigma))$, a finite set. 

  Let $U_i \subset E$ be the clopen partition induced by $\Pi$,
  and let $U_{i,j} = U_i \cap \varphi(U_j)$; let $S_i \subset
  S\setminus \Pi$ denote the complementary component containing $U_i$ and $C_i\dfn 
  \partial S_i$. Extend each component
  of $\partial \Pi$ by an embedded (but not necessarily
  essential) $\Sigma^{b+1}_{0,0}$ to obtain a subsurface
  $\Pi_\varphi \cong \Sigma$ inducing the partition $U_{i,j}$. Note that some
  components of $S \setminus \Pi_\varphi$ may be disks and that
  every component has either zero or infinite genus.  Let
  $S_{i,j} \subset S \setminus \Pi_\varphi$ denote the complementary 
  component containing $U_{i,j}$ and $C_{i,j} \dfn \partial S_{i,j}$. 

  Extend the embedding $\Pi \hookrightarrow \Pi_\varphi$ to a
  homeomorphism $\psi_\varphi : \Sigma \to \Pi_\varphi$.  Fix a
  permutation $\alpha_\varphi \in \text{Sym}(\pi_0(\partial
  \Sigma))$ and $\sigma_1(\varphi) \dfn
  \sigma_0(\alpha_\varphi)^{\psi_\varphi} : \Pi_\varphi \to
  \Pi_\varphi$
  such that $C_j' \dfn \sigma_1(\varphi)C_j$ separates $C_{i,j}$ from
  $\sigma_1(\varphi)\Pi$ for all $i$.   Let $\Pi' \dfn \sigma_1(\varphi)\Pi$ and $S_j'$ denote the component of $S \setminus \Pi'$ separated by $C_j'$.
  It follows that $\Pi' = \sigma_1(\varphi)\Pi$ induces the partition
  $\varphi(U_j)$, each of which is contained in $S'_j$. We note that
  $C'_j$ is homeomorphic to $C_j$; likewise, since the $S_j$ and
  $S_j'$ have either zero or infinite genus, $S_j$ has zero
  genus if and only if $U_j \cap \Ends_g(S) = \varnothing$ if and
  only if $\varphi(U_j) \cap \Ends_g(S) = \varnothing$ if and only
  if $S_j'$ has zero genus, and otherwise both $S_j,S_j'$ have infinite genus.
  Finally, we apply Richards' classification theorem \cite{richards} to 
  obtain $\sigma(\varphi)$ by extending
  $\sigma_1(\varphi)|_\Pi$ by homeomorphisms $\overline S_j \to
  \overline S_{j}'$
   that induce $\varphi$ on each $U_j$.  Up to isotopy $\mu \subset
   \Delta \subset \Pi$, hence
  $i(\mu,\sigma(\varphi)\mu) =
  i(\mu,\sigma_0(\alpha_\varphi)\mu)$
  is bounded
  independently of $\varphi$.
\end{proof}

\begin{figure}[h]
  \centering
  \begin{tikzpicture}
\node at (-6.1,-0.5) {$C_1$};
\node at (-1.1,-0.5) {$C_2$};
\node at (5.9,0.2) {$C_2'$};
\node at (0.9,0.2) {$C_1'$};
\node at (-4.15,-1.3) {$\Pi$};


\draw node at (-5.5,1.5) {$U_{1,1}$}
    node at (-4.3,2.1) {$U_{1,2}$}
    node at (-2.9,2.1) {$U_{2,1}$}
    node at (-1.7,2.1) {$U_{2,2}$};

\node[line width=1pt]  at (-3.5,0)
    {\includegraphics[scale=.9]{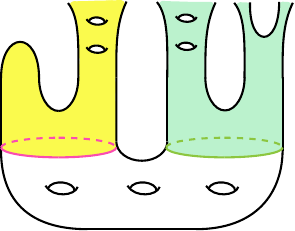}};
\node at (0,0) {$\xrightarrow{\sigma(\varphi)}$};
\node[line width=1pt]  at (3.5,0)
    {\includegraphics[scale=.9]{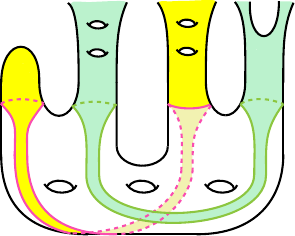}};
\node at (-1.75,0) {$S_2$};
\node at (-5.35,0) {$S_1$};
\node (u1') at (.7,1.1) {$S_1'$};
\node (p1) at (1.55,.5) {};
\draw[->] (u1') edge[bend right=20] (p1.center);

\node (pi') at (0.8,-1.5) {$\Pi'$}; 
\node (p0) at (2,-.5) {};
\draw[->] (pi') edge[bend left=20] (p0);


\node (u2'') at (1.6,1.85) {$S_2'$};
\node (p3) at (2.65, 1.2) {};
\draw[->] (u2'') edge[bend right=20] (p3.center);

  \end{tikzpicture}
  
  \caption{The map $\sigma(\varphi)$. Here, $U_{1,1}=\varnothing$}
  \label{fig:}
\end{figure}



\bibliographystyle{amsalpha}
\bibliography{pac}

\end{document}